\documentclass[12pt]{amsart}

\usepackage{amssymb}
\usepackage{amsmath}
\usepackage{amsfonts}
\usepackage{graphicx}
\usepackage{amsthm}
\usepackage{enumerate}
\usepackage[mathscr]{eucal}
\usepackage{mathrsfs}
\usepackage{verbatim}
\usepackage{booktabs}
\usepackage{enumitem}
\usepackage{color}

\makeatletter
\@namedef{subjclassname@2020}{%
  \textup{2020} Mathematics Subject Classification}
\makeatother

\numberwithin{equation}{section}

\theoremstyle{plain}
\newtheorem{theorem}{Theorem}[section]
\newtheorem{lemma}[theorem]{Lemma}

\theoremstyle{remark}

\newtheorem*{rmk}{Remark}

\allowdisplaybreaks[4]

\begin{document}

\title[Pseudo-differential operators]{$L^p$ boundedness of  pseudo-differential operators with symbols in  $S^{n(\rho-1)/2}_{\rho,1}$}


\author[J. Guo]{Jingwei Guo}
\address{Department of Mathematics\\
University of Science and Technology of China\\
Hefei 230026, China}
\email{jwguo@ustc.edu.cn}

\author[X. Zhu]{Xiangrong Zhu}
\address{Department of Mathematics\\
Zhejiang Normal University\\
Jinhua 321004, China}
\email{zxr@zjnu.cn}

\thanks{Xiangrong Zhu (the corresponding author) was supported by the National Key Research and Development Program of China (No. 2022YFA1005700) and the NSFC Grant (No. 11871436). Jingwei Guo was supported by the NSF of Anhui Province, China (No. 2108085MA12).}

\date{}

\begin{abstract}
For symbol $a\in S^{n(\rho-1)/2}_{\rho,1}$ the pseudo-differential operator $T_a$ may not be $L^2$ bounded. However, under some mild extra assumptions on $a$, we show that $T_a$ is bounded from $L^{\infty}$ to $BMO$ and on $L^p$  for $2\leq p<\infty$. A key ingredient in our proof of the $L^{\infty}$-$BMO$ boundedness is that we decompose a cube, use $x$-regularity of the symbol and combine certain $L^2$, $L^\infty$ and $L^{\infty}$-$BMO$ boundedness. We use an almost orthogonality argument to prove an $L^2$ boundedness and then interpolation to obtain the desired $L^p$ boundedness.
\end{abstract}



\maketitle

\section{Introduction and main results}

A pseudo-differential operator is an operator given by
\begin{equation}
	T_a f(x)=\int_{\mathbb{R}^n}\!e^{2\pi ix\cdot\xi}a(x,\xi)\widehat{f}(\xi)\,\textrm{d}\xi, \quad f\in \mathscr{S}(\mathbb{R}^n),  \label{1}
\end{equation}
where  $\widehat{f}$ is the Fourier transform of $f$ and  the symbol $a$ belongs to a certain symbol class. One of the most important symbol classes is the H\"{o}rmander class $S^m_{\rho,\delta}$ introduced in H\"{o}rmander \cite{H66}.
A function $a(x,\xi)\in C^{\infty}(\mathbb{R}^n\times\mathbb{R}^n)$ belongs to the H\"{o}rmander class $S^{m}_{\rho,\delta}$ $(m\in \mathbb{R},0\leq\rho,\delta\leq1)$ if it satisfies
\begin{equation*}
\sup_{x,\xi\in\mathbb{R}^{n}}(1+|\xi|)^{-m+\rho N-\delta M}\left|\nabla^{N}_{\xi}\nabla^{M}_{x}a(x,\xi)\right|=A_{N,M}<\infty
\end{equation*}
for all nonnegative integers $N$ and $M$. We may assume additionally that $a\in S^{m}_{\rho,\delta}$ is compactly supported. Since all our estimates in this paper are independent of the size of the support of symbol $a$ unless clearly stated, one can remove this extra assumption by following the argument in Stein \cite[Sec. VII.2.5]{S93}.

There are numerous literature discussing whether pseudo-differential operators are bounded on the Lebesgue space $L^p(\mathbb{R}^n)$  and the Hardy space $H^1(\mathbb{R}^n)$. We mention a few examples. If $a\in S^{m}_{\rho,\delta}$ with $\delta<1$ and $m\leq \min\{0, n(\rho-\delta)/2\}$,
then $T_a$ is bounded on $L^2$ and the range of $m$ is sharp. See H\"{o}rmander \cite{H71}, Calder\'{o}n-Vaillancourt \cite{CV71,CV72}, Hounie \cite{H86}, etc. For $a\in S^{m}_{\rho,1}$, Rodino \cite{R76} proved that
$T_a$ is bounded on $L^2$ if $m<n(\rho-1)/2$ and constructed a symbol $a\in S^{n(\rho-1)/2}_{\rho,1}$ such that $T_a$ is unbounded on $L^2$. For endpoint estimates, in some unpublished lecture notes,
Stein showed that if $a\in S^{n(\rho-1)/2}_{\rho,\delta}$ and either $0\leq\delta<\rho=1$ or $0<\delta=\rho<1$, then $T_a$ is of weak type $(1,1)$ and bounded from $H^1$ to $L^1$.
This result was extended in \'{A}lvarez-Hounie \cite{AH90} to symbols $a\in S^{m}_{\rho,\delta}$ with $0<\rho\leq 1, 0\leq \delta<1$ and $m=n(\rho-1+\min\{0,\rho-\delta\})/2$. For a systematic study on the $H^1$-$L^1$ boundedness of $T_a$ with $a\in S^{m}_{\rho,1}$ when $m$ equals to the critical index $n(\rho-1)$, see the authors \cite{GZ}.

On the other hand, Kenig-Staubach \cite{KW07} proved that $T_a$ is bounded on $L^{\infty}$ if $a\in S^{m}_{\rho,1}$ with $m<n(\rho-1)/2$. It is also known that even if $a\in S^{n(\rho-1)/2}_{\rho,0}$ is independent of $x$,  $T_a$ is still not bounded on $L^{\infty}$ in general (see \cite[Remark 2.6]{KW07}).

When $a\in S^{n(\rho-1)/2}_{\rho,1}$, as $T_a$ is not bounded on $L^2$ in general, it is reasonable to expect that $T_a$ is not bounded from $L^{\infty}$ to $BMO$
in general either. However, under some mild extra assumptions we will show that it is bounded from $L^{\infty}$ to $BMO$ and on $L^p$ as well for $2\leq p<\infty$.
\begin{theorem}\label{main3}
	Suppose that $w$ is a function from $(0,\infty)$ to $(0,\infty)$ and there exist constants $A>1$ and $u\in (0,n/(n+2))$ such that
\begin{equation}
\int^{\infty}_1\left(\frac{w(t)}{t}\right)^{u}\frac{\textrm{d}t}{t}\leq A \label{6}	
\end{equation}
and $w(t_1)\leq Aw(t_2)$ whenever $0<t_1\leq t_2$.

If the symbol $a\in S^{n(\rho-1)/2}_{\rho,1}$, $0\leq\rho<1$, satisfies that
\begin{equation*}
\sup_{x,\xi\in\mathbb{R}^{n}} \left(1+|\xi|\right)^{-n(\rho-1)/2+\rho N}\left(w\left(|\xi|\right)\right)^{-1}\left|\nabla^{N}_{\xi}\nabla_{x}a(x,\xi)\right|=A_{N}<\infty
\end{equation*}
for all nonnegative integer $N$, then $T_a$ is bounded from $L^{\infty}(\mathbb{R}^n)$ to $BMO(\mathbb{R}^n)$.
\end{theorem}

\begin{rmk}
The condition $\rho<1$ is assumed here for some technical reason. When $\rho=1$, this theorem still holds  (and we actually have a better version). See the last remark of this section.
\end{rmk}

\begin{rmk}
Our theorem works for a class of general functions $w$. It include some known results in literature corresponding to specific $w$, for example, in \cite{AH90} (with $w(t)=t^{\delta}$, $0\leq \delta\leq \rho$ and $\delta<1$),  \cite{W, WC} (with $w(t)=t^{\delta}$,  $0\leq\delta<1$) and \cite{RZ23} (with $w(t)=t/\log^6 (1+t)$). It also works for some ``weaker'' functions, for example $w(t)=t/\log^{\frac{n+2}{n}+\epsilon}(1+t)$  with $\epsilon>0$.
\end{rmk}


\begin{rmk}
The idea to prove this theorem is as follows.  To estimate the $BMO$ norm of $T_a f$, we estimate $\inf_c \frac{1}{|Q|}\int_Q\!|T_a f(x)-c|\textrm{d}x$ for any cube $Q$. We first make a standard decomposition of  $T_a f(x)$ on the frequency side.  Depending on the size of frequency, the side length of $Q$ and the function $w$, we may further decompose the cube $Q$ into an almost disjoint union of smaller cubes of the same side length, $Q=\cup_{k=1}^{K} Q_{j,k}$.  Over each $Q_{j,k}$,  we localize the $x$-variable of symbol $a(x,\xi)$ to the center of $Q_{j,k}$ and split the function $f$ into two parts with one restricted to an enlarged concentric cube $\widetilde{Q}_{j,k}$ and another restricted to the complement of $\widetilde{Q}_{j,k}$. With all these decompositions, the average over $Q$ is split into four parts. To estimate them, we use the $x$-regularity assumption on symbol $a$, certain known $L^2$ and $L^{\infty}$-$BMO$ boundedness and an $L^\infty$ boundedness we prove in Lemma \ref{l2.3}. At last we balance bounds of different parts to obtain the desired one.
\end{rmk}

\begin{rmk}
Only finitely many derivatives of $a$ are needed in this theorem. In this paper we do not pursue the best order of derivatives needed to guarantee the conclusion, which depends on $n$ and $u$.
\end{rmk}

By an almost orthogonality argument we obtain the following.

\begin{theorem}\label{main1}
	Suppose that $w$ is a function from $(0,\infty)$ to $(0,\infty)$ and there exists a constant $A>1$ such that
	\begin{equation}
	\int^{\infty}_1\frac{w^2(t)}{t^3}\,\textrm{d}t\leq A \label{2}
	\end{equation}
and $w(t_1)\leq Aw(t_2)$ whenever $0< t_1\leq t_2$.

If the symbol $a\in S^{n(\rho-1)/2}_{\rho,1}$, $0\leq\rho\leq 1$, satisfies that
	\begin{equation*}
	\sup_{x,\xi\in\mathbb{R}^{n}}(1+|\xi|)^{-n(\rho-1)/2+\rho N}\left(w\left(|\xi|\right)\right)^{-1}\left|\nabla^{N}_{\xi}\nabla_{x}a(x,\xi)\right|=A_{N}<\infty
	\end{equation*}
	for all nonnegative integer $N\leq n$, then $T_a$ is bounded on $L^{2}(\mathbb{R}^n)$.
\end{theorem}

Since \eqref{6} implies that $w(t)/t$ is uniformly bounded when $t\geq 1$ (see inequality \eqref{w1} below), assumptions of Theorem \ref{main3} imply those of Theorem \ref{main1}. Applying interpolation and Theorems \ref{main3} and \ref{main1}, we obtain the following.
\begin{theorem}\label{main2}
Under assumptions of Theorem \ref{main3}, $T_a$ is bounded on $L^p(\mathbb{R}^n)$ for $2\leq p<\infty$.
\end{theorem}

\begin{rmk}
When $a\in S^0_{1,1}$, it is well-known that the integral kernel
\begin{equation*}
	k(x,y)=\int_{\mathbb{R}^n}\!e^{2\pi i(x-y)\cdot \xi}a(x,\xi)\,\textrm{d}\xi
\end{equation*}
associated to $T_a$ satisfies the H\"{o}rmander condition. So the $L^2$ boundedness of $T_a$ yields the $L^{\infty}$-$BMO$, $H^1$-$L^1$ and $L^1$-$L^{1,\infty}$ boundedness directly by the standard
Calder\'{o}n-Zygmund theory. Therefore, when $\rho=1$ under assumptions of Theorem \ref{main1} (or stronger assumptions of Theorem \ref{main3}) we obtain conclusions of Theorems \ref{main3} and \ref{main2} immediately.
\end{rmk}

Throughout this note, we use $C$ to denote a positive constant, which may vary from line to line and depend only on $n,\rho,u,A$ and finitely many seminorms of $a$.


\section{Preliminaries}\label{s2}

We first recall two fundamental results on pseudo-differential operators.

\begin{lemma}[H\"{o}rmander \!\cite{H71}, Calder\'{o}n-Vaillancourt \!\cite{CV71}, Hounie \!\cite{H86}]  \label{l2.1}
		If\\
		 $0\leq\rho\leq 1$ and $a\in S^{0}_{\rho,0}$, then
	\begin{equation*}
		\|T_{a}f\|_{2}\leq C\|f\|_{2},
	\end{equation*}
	where the constant $C$ depends only on $n$, $\rho$ and finitely many seminorms of $a$ in $S^{0}_{\rho,0}$.
\end{lemma}

\begin{lemma}[{\'{A}lvarez-Hounie \cite{AH90}, Stein \cite[p. 322, Sec. VII.5.12(h)]{S93}}]\label{l2.2}
	If  $0\leq \rho\leq 1$ and $a\in S^{n(\rho-1)/2}_{\rho,0}$, then
	$$\|T_{a}f\|_{BMO}\leq C\|f\|_{\infty},$$
	where the constant $C$ depends only on $n,\rho$ and	 finitely many seminorms of $a$ in $S^{n(\rho-1)/2}_{\rho,0}$.
\end{lemma}

We will use the well-known Littlewood-Paley dyadic decomposition. Let $B_r$ be the ball in $\mathbb{R}^n$ centered at the origin with radius $r$. Take a nonnegative function $\eta\in C^{\infty}_{c}(B_{2})$ with $\eta\equiv 1$ on $B_{1}$ and
set $\varphi(\xi)=\eta(\xi)-\eta(2\xi)$. It is obvious that $\varphi$ is supported in $\{\xi\in \mathbb{R}^n: 1/2<|\xi|<2\}$ and
$$\eta(\xi)+\sum^{\infty}_{j=1}\varphi(2^{-j}\xi)=1, \textrm{ for all $\xi\in \mathbb{R}^n$}.$$
We denote functions $\varphi_0(\xi)=\eta(\xi)$ and $\varphi_j(\xi)=\varphi(2^{-j}\xi)$ for $j\geq 1$, and operators $\triangle_j$ and $S_j$ by
\begin{equation}\label{pddbmo2.1}
	\widehat{\triangle_jf}(\xi)=\varphi_j(\xi)\widehat{f}(\xi) \textrm{ and } S_jf=\sum^j_{k=0}\triangle_kf \textrm{ for } j\geq 0.
\end{equation}
We also use the following notations
\begin{equation}
	a_j(x,\xi)=a(x,\xi)\varphi_j(\xi) \label{5}
\end{equation}
and
\begin{equation}
T_{a,j}f(x)=T_{a}(\triangle_jf)(x)=\int \! e^{2\pi ix\cdot\xi}a(x,\xi)\varphi_j(\xi)\widehat{f}(\xi)\textrm{d}\xi \textrm{ for } j\geq 0.\label{3}
\end{equation}

For the pseudo-differential operator, we have the following estimate.
\begin{lemma}\label{l2.3}
	Let $n\in \mathbb{N}$, $0\leq\rho\leq 1$ and $j$ a nonnegative integer. If the symbol $b(x,\xi)$ is supported in $\{(x,\xi)\in\mathbb{R}^n\times \mathbb{R}^n:|\xi|<2^{j+1}\}$ and satisfies that
	\begin{equation*}
	\sup_{x,\xi\in\mathbb{R}^{n}}\left|\partial^{\alpha}_{\xi}b(x,\xi)\right|\leq A_j2^{-j\rho|\alpha|} 
	\end{equation*}
	for any multi-index $\alpha$ with $|\alpha|\leq n$, then we have
\begin{equation*}
\|T_b f\|_{p}\leq CA_j2^{jn(1-\rho)/2}\|f\|_{p}, \quad 2\leq p\leq \infty,
\end{equation*}
where the constant $C$ depends only on $n$.
\end{lemma}

\begin{proof}
Set $\sigma_j(z)=2^{jn\rho}(1+2^{j\rho}|z|)^{-2n}$. By the Cauchy-Schwarz inequality, we have
\begin{align*}
     &|T_b f(x)|\\
\leq &\left(\int\!\left(1+2^{j\rho}|x-y|\right)^{-2n}|f(y)|^2\textrm{d}y\right)^{1/2}\\
&\left(\int\!\left|\left(1+2^{j\rho}|x-y|\right)^{n}\int\! e^{2\pi i(x-y)\cdot \xi} b(x,\xi)\textrm{d}\xi\right|^2\textrm{d}y\right)^{1/2}\\
\leq & C_n2^{-\frac{jn\rho}{2}}\!\!\left(\sigma_j\ast|f|^2\right)^{\frac{1}{2}}\!(x)\!\!\sum_{|\alpha|\leq n}\!\!\left(\!\int\!\left|\int\! e^{2\pi i(x-y)\cdot \xi}2^{j\rho|\alpha|}\partial_\xi^{\alpha}b(x,\xi)\textrm{d}\xi\right|^2\!\!\!\textrm{d}y\!\right)^{\!\!\!1/2}\!\!.
\end{align*}
By the Plancherel theorem and assumptions,  the right side is bounded by
\begin{align*}
&C_n 2^{-\frac{jn\rho}{2}}\left(\sigma_j\ast|f|^2\right)^{1/2}(x)\sum_{|\alpha|\leq n}\left(\int \left|2^{j\rho|\alpha|}\partial_\xi^{\alpha}b(x,\xi)\right|^2\textrm{d}\xi\right)^{1/2}\\
\leq &C_n 2^{\frac{jn(1-\rho)}{2}}A_j\left(\sigma_j\ast|f|^2\right)^{1/2}(x).
\end{align*}
By Young's inequality, for $2\leq p\leq \infty$, we get that
\begin{equation*}
\|T_b f\|_p\leq C_n 2^{\frac{jn(1-\rho)}{2}}A_j\left\|(\sigma_j\ast|f|^2)^{1/2}\right\|_p
\leq  C_nA_j 2^{\frac{jn(1-\rho)}{2}}\|f\|_p.
\end{equation*}
This finishes the proof.
\end{proof}


\section{Proof of Theorem \ref{main3}}

For any cube $Q$, it is enough for us to choose a constant $\lambda_Q$ such that
\begin{equation*}
	\frac{1}{|Q|}\int_{Q}|T_{a}f(x)-\lambda_Q|\,\textrm{d}x\leq C\|f\|_{\infty}
\end{equation*}
for some constant $C$ independent of $Q$. Let $l(Q)$ be the side length of $Q$ and $x_Q$ the center of $Q$.

We first decompose $T_{a}f(x)$ into two parts
\begin{align}
	T_{a}f(x)&=\int e^{2\pi ix\cdot\xi}a(x,\xi)\left(\sum^{j_Q}_{j=0}\varphi_j(\xi)+\sum^{\infty}_{j=j_Q+1}\varphi_j(\xi)\right)\widehat{f}(\xi)\,\textrm{d}\xi \nonumber\\
	        &=T_{a}(S_{j_Q}f)(x)+\sum^{\infty}_{j=j_Q+1} T_{a,j}f(x),\label{lppdo5.1}
\end{align}	
where operators $S_{j_Q}$ and $T_{a,j}$ are defined by \eqref{pddbmo2.1} and \eqref{3}. We take $j_Q=\infty$ if $l(Q)\int^{\infty}_1\frac{w(t)}{t}dt\leq 1$ (that is, \eqref{lppdo5.1} only has the first part)
 and $j_Q=-1$ if $l(Q)\int^4_1\frac{w(t)}{t}dt>1$ (that is, \eqref{lppdo5.1} only has the second part). Otherwise, we take $j_Q$ to be the unique nonnegative integer satisfying
\begin{equation*}
	\int^{2^{j_Q+2}}_1\frac{w(t)}{t}\,\textrm{d}t\leq \frac{1}{l(Q)}<\int^{2^{j_Q+3}}_1\frac{w(t)}{t}\,\textrm{d}t.
\end{equation*}

\textbf{Step 1.} In this step we treat the first part of \eqref{lppdo5.1}. Set
\begin{equation*}
\widetilde{T}_{a_Q}f(x)=\int_{\mathbb{R}^{n}}e^{2\pi ix\cdot\xi}a(x_Q,\xi)\widehat{f}(\xi)\,\textrm{d}\xi.
\end{equation*}
	Because $a_Q:=a(x_Q,\xi)\in S^{n(\rho-1)/2}_{\rho,0}$ and its semi-norms are independent of $Q$, Lemma \ref{l2.2} gives that
\begin{equation*}
\left\|\widetilde{T}_{a_Q}f\right\|_{BMO}\leq C\|f\|_{\infty}.
\end{equation*}
	It is easy to see that $\|S_{j_Q}f\|_{\infty}\leq C\|f\|_{\infty}$ which yields that
\begin{equation*}
\left\|\widetilde{T}_{a_Q}(S_{j_Q}f)\right\|_{BMO}\leq C\|S_{j_Q}f\|_{\infty}\leq C\|f\|_{\infty}.
\end{equation*}
	Thus we can choose a constant $\lambda_Q$ such that
	\begin{align}
		\frac{1}{|Q|}\int_{Q}\left|\widetilde{T}_{a_Q}(S_{j_Q}f)(x)-\lambda_Q\right|\,\textrm{d}x\leq C\|f\|_{\infty} \label{lppdo5.2}
	\end{align}
with a constant $C$ independent of $Q$.
	
Set
\begin{equation*}
b_{Q}(x,\xi)=\frac{a(x,\xi)-a(x_Q,\xi)}{l(Q)}\eta\left(\frac{x-x_Q}{nl(Q)}\right).
\end{equation*}
It is easy to verify that $b_{Q}(x,\xi)=0$ when $|x-x_Q|\geq 2nl(Q)$ and that $b_{Q}(x,\xi)=(a(x,\xi)-a(x_Q,\xi))/l(Q)$ when $x\in Q$. By using assumptions on $a$ and $w$, we have for any multi-index $\alpha$ that
\begin{equation*}
	\left|\partial^{\alpha}_{\xi}\left(b_{Q}(x,\xi)\varphi_j(\xi)\right)\right|\leq C2^{j(\frac {n(\rho-1)}{2}-\rho |\alpha|)}w\left(2^{j+1}\right).
\end{equation*}
Hence, by using Lemma \ref{l2.3} (with $p=\infty$), we obtain that
	\begin{equation*}
		\left\|T_{b_{Q},j}f\right\|_{\infty}\leq Cw\left(2^{j+1}\right)\|f\|_{\infty}. 
	\end{equation*}
Thus if $x\in Q$ then
	\begin{align*}
		&\left|T_{a}(S_{j_Q}f)(x)-\widetilde{T}_{a_Q}(S_{j_Q}f)(x)\right|\\
	\leq & l(Q)\sum^{j_Q}_{j=0}\left|T_{b_{Q},j}f(x)\right|\leq C l(Q)\sum^{j_Q}_{j=0} w\left(2^{j+1}\right)\|f\|_{\infty}.
	\end{align*}
	
Applying the triangle inequality, the bound \eqref{lppdo5.2} and the above bound yields that
	\begin{align}
		&\frac{1}{|Q|}\int_Q\left|T_{a}(S_{j_Q}f)(x)-\lambda_Q\right|\,\textrm{d}x\nonumber\\
		\leq &\frac{1}{|Q|}\int_Q\left|T_{a}(S_{j_Q}f)(x)-\widetilde{T}_{a_Q}(S_{j_Q}f)(x)\right|\,\textrm{d}x+C\|f\|_{\infty}\nonumber\\
		\leq &C\left(l(Q)\sum^{j_Q}_{j=0}w\left(2^{j+1}\right)+1\right)\|f\|_{\infty}\nonumber\\
		\leq &C\left(l(Q)\int^{2^{j_Q+2}}_1 \frac{w(t)}{t}\,\textrm{d}t+1\right)\|f\|_{\infty}\nonumber\\
		\leq &C\|f\|_{\infty},\label{lppdo5.4}
	\end{align}
where in the last inequality we have used the definition of $j_Q$.

\textbf{Step 2.}  From this step we start to treat the second part of \eqref{lppdo5.1}.  For each $j>j_Q$, we decompose the cube $Q$ into an almost disjoint union of finitely many cubes,  $Q=\cup_{k=1}^{K} Q_{j,k}$, such that
\begin{equation*}
l_j/C\leq l(Q_{j,k})\leq l_j \textrm{ and } \ K\leq Cl^{-n}_j|Q|,
\end{equation*}
where $l_j=2^{-ju}w(2^{j+2})^{u-1}$ and $l(Q_{j,k})$ represents the side length of $Q_{j,k}$. This is feasible because by using assumptions on $w$, we get
	\begin{align*}
 &\int^{2^{j+2}}_1\frac{w(t)}{t}\,\textrm{d}t \leq A^{1-u}2^{(j+2)u}w\left(2^{j+2}\right)^{1-u}\int^{2^{j+2}}_1\!\left(\frac{w(t)}{t}\right)^u \frac{\textrm{d}t}{t}\\
 \leq &C_{A,u}2^{ju}w\left(2^{j+2}\right)^{1-u}=C_{A,u} l^{-1}_j,
	\end{align*}
which implies that $l_j\leq C l(Q)$ if $j>j_Q$.

Let $x_{j,k}$ be the center of $Q_{j,k}$, $a_{Q_{j,k}}\!\!:=a(x_{j,k},\xi)$ and $\widetilde{T}_{a_{Q_{j,k}, j}}$, $b_{Q_{j,k}}$ be defined as above. Then, by using $l(Q_{j,k})\leq l_j$ and Lemma \ref{l2.3} (with $p=\infty$), we readily obtain, as in Step 1, that  
\begin{align}
		&\frac{1}{|Q|}\sum^K_{k=1}\int_{Q_{j,k}}\left|T_{a,j}f(x)-\widetilde{T}_{a_{Q_{j,k}},j}f(x)\right|\,\textrm{d}x\nonumber\\
		= & \frac{1}{|Q|}\sum^K_{k=1}l(Q_{j,k})\int_{Q_{j,k}}\left|T_{b_{Q_{j,k}},j}f(x)\right|\,\textrm{d}x\nonumber\\
	\leq & \frac{l_j}{|Q|}\sum^K_{k=1}\left|Q_{j,k}\right|\left\|T_{b_{Q_{j,k}},j}f\right\|_{\infty}\nonumber\\
	\leq & Cl_jw\left(2^{j+1}\right)\|f\|_{\infty}\leq C\left(2^{-j}w\left(2^{j+2}\right)\right)^u\|f\|_{\infty},\label{lppdo5.5}
\end{align}
where $C$ is independent of $Q$.

	\textbf{Step 3.} Set $L_j=2+2^{j(1-\rho)}(w(2^{j+2})/2^j)^{2u/n}$. Let
\begin{equation*}
\widetilde{Q}_{j,k}=Q(x_{j,k},L_j l(Q_{j,k}))
\end{equation*}
be an enlarged $Q_{j,k}$ with center $x_{j,k}$ and side length $L_j l(Q_{j,k})$, and
\begin{equation*}
f_{j,k}=f\chi_{\widetilde{Q}_{j,k}}
\end{equation*}
be $f$ restricted to $\widetilde{Q}_{j,k}$.
	
One can easily check that $2^{jn(1-\rho)/2}a(x_{j,k},\xi)\varphi_j(\xi)\in S^{0}_{\rho,0}$ and its seminorms
	are independent of $j$, $k$ and $Q$. Hence, by Lemma \ref{l2.1}, we get that
	\begin{align*}
		\left\|\widetilde{T}_{a_{Q_{j,k}},j}f_{j,k}\right\|_2&=\left\|2^{\frac{jn(\rho-1)}{2}}\int_{\mathbb{R}^{n}}e^{2\pi ix\cdot\xi}2^{\frac{jn(1-\rho)}{2}}a(x_{j,k},\xi)\varphi_j(\xi)\widehat{f_{j,k}}(\xi)\,\textrm{d}\xi\right\|_2\\
		&\leq C2^{\frac{jn(\rho-1)}{2}}\left\|f_{j,k}\right\|_2\\
		&\leq C2^{\frac{jn(\rho-1)}{2}}L_j^{\frac n2}|Q_{j,k}|^{\frac 12}\|f\|_{\infty},
	\end{align*}
	where $C$ is independent of $j$, $k$ and $Q$. Using this bound and H\"{o}lder's inequality, we have
	\begin{align}
		&\frac{1}{|Q|}\sum^K_{k=1}\int_{Q_{j,k}}\!\left|\widetilde{T}_{a_{Q_{j,k}},j}f_{j,k}(x)\right|\textrm{d}x\nonumber\\
		\leq &\frac{1}{|Q|}\sum^K_{k=1}\left|Q_{j,k}\right|^{\frac 12}\left\|\widetilde{T}_{a_{Q_{j,k}},j}f_{j,k}\right\|_2 \leq C 2^{\frac{jn(\rho-1)}{2}}L_j^{\frac n2}\|f\|_{\infty}\nonumber\\
		\leq &C\left(\left(\frac{w(2^{j+2})}{2^j}\right)^{u}+2^{\frac{jn(\rho-1)}{2}}\right)\|f\|_{\infty}.\label{lppdo5.8}
	\end{align}

	\textbf{Step 4.} As $u<n/(n+2)$, we can take $N>n$ sufficiently large such that
	\begin{equation*}
	\left(N-\frac n2\right)\left(1-\frac{n+2}{n}u\right)\geq u.
	\end{equation*}
	By using assumptions on $w$, we have that for $t\geq 1$
\begin{align}\label{w1}
	\left(\!\frac{w(t)}{t}\!\right)^{\!\!u}=\int^{2t}_t \!\left(\!\frac{w(t)}{t}\!\right)^{\!\!u}\frac{\textrm{d}s}{t}\leq 2^{1+u}A^{u}\!\int^{2t}_t \!\left(\!\frac{w(s)}{s}\!\right)^{\!\!u}\frac{\textrm{d}s}{s}\leq (2A)^{1+u}.
\end{align}
	So, for any nonnegative integer $j$, one has
	\begin{align}\label{w2}
	(1+2^{j\rho}L_jl_j)^{\frac n2-N}\leq \left(\frac{w(2^{j+2})}{2^j}\right)^{\!\left(\frac{n+2}{n}u-1\right)\left(\frac n2-N\right)}\!\!\!\leq C\left(\frac{w(2^{j+2})}{2^j}\right)^{\!u}.
	\end{align}
	As $L_j>2$, when $y\notin \widetilde{Q}_{j,k}$ and $x\in Q_{j,k}$, one also has
\begin{equation*}
|y-x|\geq
\frac{1}{2}(L_j-1)l(Q_{j,k})\geq \frac{1}{4}L_j l(Q_{j,k})\geq c L_j l_j.
\end{equation*}

	For $x\in Q_{j,k}$, applying the Cauchy-Schwarz inequality yields that
	\begin{align*}
		&\left|\widetilde{T}_{a_{Q_{j,k}},j}\left(f-f_{j,k}\right)(x)\right|\nonumber\\
		=&\left|\int_{\widetilde{Q}^c_{j,k}}\left(\int e^{2\pi i(x-y)\cdot \xi} a\left(x_{j,k},\xi\right)\varphi_j(\xi)\textrm{d}\xi\right)f(y)\,\textrm{d}y\right|\nonumber\\
	\leq &  \left(\int_{|y-x|\geq cL_jl_j}\left(1+2^{j\rho}|x-y|\right)^{-2N}\textrm{d}y\right)^{1/2}\nonumber\\
		&\left(\int \!\!\left(\left(1+2^{j\rho}|x-y|\right)^{N}\!\!\!\int e^{2\pi i(x-y)\cdot \xi} a(x_{j,k},\xi)\varphi_j(\xi)\textrm{d}\xi\right)^2\!\!\!\textrm{d}y\right)^{\!\!\!1/2}\!\!\!\!\!\|f\|_{\infty},
	\end{align*}
where by \eqref{w2} the first factor is bounded by
\begin{equation*}
	C2^{-\frac{jn\rho}{2}}\left(1+2^{j\rho} L_jl_j\right)^{\frac n2-N}\leq C2^{-\frac{jn\rho}{2}}\left(\frac{w(2^{j+2})}{2^j}\right)^{u},
\end{equation*}
and, by the Plancherel theorem and $a\in S^{n(\rho-1)/2}_{\rho,1}$, the second factor is bounded by
		\begin{align*}
		&C\sum_{|\alpha|\leq N}\left(\int \left|\int e^{2\pi i(x-y)\cdot \xi} 2^{j\rho|\alpha|}\partial^{\alpha}_{\xi}\left(a(x_{j,k},\xi)\varphi_j(\xi)\right)\textrm{d}\xi\right|^2\textrm{d}y\right)^{1/2}\\
		\leq &C\sum_{|\alpha|\leq N}\left(\int
		\left|2^{j\rho|\alpha|}\partial^{\alpha}_{\xi}(a(x_{j,k},\xi)\varphi_j(\xi))\right|^2\textrm{d}\xi\right)^{1/2}\\
		\leq &C\sum_{|\alpha|\leq N}\left(\int_{2^{j-1}<|\xi|<2^{j+1}}2^{jn(\rho-1)}\textrm{d}\xi\right)^{1/2}\leq C 2^{\frac{jn\rho}{2}}.
	\end{align*}
Hence we obtain
\begin{equation}
\left|\widetilde{T}_{a_{Q_{j,k}},j}\left(f-f_{j,k}\right)(x)\right| \leq C\left(\frac{w(2^{j+2})}{2^j}\right)^{u}\|f\|_{\infty}.\label{lppdo5.10}
\end{equation}

	\textbf{Step 5.}  	Finally, we infer from \eqref{lppdo5.1}, \eqref{lppdo5.4}, \eqref{lppdo5.5}, \eqref{lppdo5.8} and \eqref{lppdo5.10} that
	\begin{align*}
		&\frac{1}{|Q|}\int_{Q}\left|T_{a}f(x)-\lambda_Q\right|\,\textrm{d}x\\
		\leq &\frac{1}{|Q|}\int_{Q}\left|T_{a}(S_{j_Q}f)(x)-\lambda_Q\right|+\sum^{\infty}_{j=j_Q+1} \left|T_{a,j}f(x)\right|\,\textrm{d}x\\
		\leq &C\|f\|_{\infty}+\frac{1}{|Q|}\sum^{\infty}_{j=j_Q+1}\sum^K_{k=1}\int_{Q_{j,k}}\bigg(\left|T_{a,j}f(x)-\widetilde{T}_{a_{Q_{j,k}},j}f(x)\right|\\
		&+\left|\widetilde{T}_{a_{Q_{j,k}},j}f_{j,k}(x)\right|+\left|\widetilde{T}_{a_{Q_{j,k}},j}(f-f_{j,k})(x)\right|\bigg)\,\textrm{d}x\\
		\leq &C\bigg(1+\sum^{\infty}_{j=j_Q+1}\left(\left(\frac{w(2^{j+2})}{2^j}\right)^{u}+2^{\frac{jn(\rho-1)}{2}}\right)\bigg)\|f\|_{\infty}\\
		\leq & C\left(1+\int^{\infty}_{2^{j_Q}}\left(\frac{w(t)}{t}\right)^{u}\frac{\textrm{d}t}{t}\right)\|f\|_{\infty}\leq C \|f\|_{\infty},
	\end{align*}
where we have used  assumptions $\rho<1$ and  \eqref{6}.	This completes the proof of Theorem \ref{main3}. \qed


\section{Proof of Theorem \ref{main1}}

We choose a smooth real function $\psi$ such that
\begin{equation*}
	\widehat{\psi}\in C_c^{\infty}(B_{1/100}) \textrm{ and } \int_{\mathbb{R}^n}\!\psi(x)\,\textrm{d}x=1.
\end{equation*}
We then decompose $a$ as follows
\begin{align*}
	a(x,\xi)=&\sum^{\infty}_{j=0}a(x,\xi)\varphi_j(\xi)\\
	=&\int\sum^{\infty}_{j=0} a_j(x-u,\xi)2^{jn}\psi\left(2^ju\right)\textrm{d}u\\
	&+\sum^{\infty}_{j=0}\int \left(a_j(x,\xi)-a_j(x-u,\xi)\right)2^{jn}\psi\left(2^ju\right)\textrm{d}u\\
	=:&b(x,\xi)+\sum^{\infty}_{j=0}\widetilde{a}_j(x,\xi).
\end{align*}
Hence $T_a=T_b+\sum^{\infty}_{j=0}T_{\widetilde{a}_j}$.

Note that $\xi$-support of the symbol $\widetilde{a}_j$ is contained in $\{\xi\in \mathbb{R}^n:|\xi|<2^{j+1}\}$ and
\begin{equation*}
\left|\partial^{\alpha}_{\xi}\widetilde{a}_j(x,\xi)\right|\leq C w\left(2^{j+1}\right)2^{j(\frac{n(\rho-1)}{2}-1-\rho|\alpha|)}.
\end{equation*}
Thus by Lemma \ref{l2.3} (with $p=2$) we get
\begin{equation*}
\|T_{\widetilde{a}_j}f\|_2\leq C w\left(2^{j+1}\right) 2^{-j}\|f\|_2.
\end{equation*}
To keep the restriction on the frequency of $\widetilde{a}_j$, we introduce multiplier operators
\begin{equation*}
\widehat{\triangle^{\prime}_0f}(\xi)=\textbf{1}_{\{\xi : |\xi|\leq 2\}}(\xi)\widehat{f}(\xi)
\end{equation*}
and
\begin{equation*}
	\widehat{\triangle^{\prime}_j f}(\xi)=\textbf{1}_{\{\xi : 2^{j-1}\leq |\xi|\leq 2^{j+1}\}}(\xi)\widehat{f}(\xi), \quad j\in\mathbb{N}.
\end{equation*}
Therefore,
\begin{align*}
	\left\|\sum^{\infty}_{j=0}T_{\widetilde{a}_j}f\right\|_{2}\leq & \sum^{\infty}_{j=0}\left\|T_{\widetilde{a}_j}f\right\|_2=\sum^{\infty}_{j=0}\left\|T_{\widetilde{a}_j}\left( \triangle^{\prime}_j f\right)\right\|_2\nonumber\\
	\leq &C\sum^{\infty}_{j=0}w\left(2^{j+1}\right)2^{-j}\left\|\triangle^{\prime}_j f\right\|_2\nonumber\\
	\leq &C \left(\sum^{\infty}_{j=0}w^2(2^{j+1})2^{-2j} \right)^{1/2} \left(\sum^{\infty}_{j=0}\left\|\triangle^{\prime}_j f\right\|_2^2 \right)^{1/2}\nonumber\\
	\leq &C\int^{\infty}_1\frac{w^2(t)}{t^3}\,\textrm{d}t \|f\|_2 \leq C\|f\|_2.
\end{align*}

It remains to handle the operator $T_b$. We observe that $T_b f=S(\widehat{f})$ with
\begin{equation*}
Sf(x)=\int_{\mathbb{R}^n}\!e^{2\pi i x\cdot\xi}b(x,\xi)f(\xi)\,\textrm{d}\xi.
\end{equation*}
By the Plancherel theorem, it suffices to establish the $L^2$ boundedness of the operator $S$. Let $S^*$ be the adjoint operator of $S$. It then suffices to show the $L^2$ boundedness of
	\begin{align*}
		SS^*f(x)&=\iint e^{2\pi i (x-y)\cdot\xi}b(x,\xi)\overline{b(y,\xi)}f(y)\,\textrm{d}y\textrm{d}\xi\\
		&=\sum_{j=0}^{\infty} \sum_{\substack{|k-j|\leq 1\\ k\geq 0}} \iint e^{2\pi i (x-y)\cdot\xi}a_j\ast_1\!\Psi_j (x,\xi)\overline{a_k}\ast_1 \!\Psi_k(y,\xi)f(y)\,\textrm{d}y\textrm{d}\xi\\
		&=: \sum^{\infty}_{j=0}R_jf(x),
	\end{align*}
where we denote $\Psi_j(u)=2^{jn}\psi(2^j u)$, $j\in\mathbb{N}$, and a partial convolution $f\!*_1\!\phi (x,\xi)=\int_{\mathbb{R}^n} \! f(x-u,\xi)\phi(u)\,\textrm{d}u$ for functions $f$ on $\mathbb{R}^{2n}$ and $\phi$ on $\mathbb{R}^n$.

We observe that
	\begin{align}\label{lppdo3.1}
	R_jR_k^*=R^*_jR_k=0 \textrm{ if }|j-k|\geq 5,
\end{align}	
where $R_j^*$ denotes the adjoint operator of $R_j$. Indeed, notice that
	\begin{align*}
	&\widehat{R_jf}(\eta)=\\
   &\sum_{\substack{|k-j|\leq 1\\ k\geq 0}}\!\int \!\! e^{2\pi i (z\cdot(\xi-\eta)-y\cdot\xi)}a_j(z,\xi)\widehat{\psi}\left(2^{-j}(\eta-\xi)\right)\overline{a_k}\ast_1 \!\Psi_k(y,\xi)f(y)\textrm{d}z\textrm{d}y\textrm{d}\xi.
    \end{align*}
Hence $\widehat{R_jf}(\eta)\neq 0$ implies that there exists a point $\xi$ such that
\begin{equation*}
a_j(z,\xi)\widehat{\psi}\left(2^{-j}(\eta-\xi)\right)\neq 0.
\end{equation*}
It is easy to see that $\widehat{R_j f}$ is supported in
	$\{\eta:|\eta|\leq \frac{201}{100}\}$ if $j=0$, and in the shell $\frac{49}{100}2^{j}\leq |\eta|\leq \frac{201}{100}2^{j}$ if $j\geq 1$. Similar computation shows that $\widehat{R_j^*f}$ is supported in 	$\{\eta:|\eta|\leq \frac{201}{100}2^{j+1}\}$ if $j=0,1$, and in the shell $\frac{49}{100}2^{j-1}\leq |\eta|\leq \frac{201}{100}2^{j+1}$ if $j\geq 2$. By the Plancherel theorem we have
\begin{align*}
	\langle R_j^*R_kf, g\rangle=\langle R_k f, R_j g\rangle=\langle \widehat{R_k f}, \widehat{R_j g}\rangle.
\end{align*}	
Thus we get \eqref{lppdo3.1}.

We also observe that
\begin{align}\label{lppdo3.2}
	\|R_jf\|_2\leq C\|f\|_2.
\end{align}
Indeed, if we denote
\begin{equation*}
Q_j f(x)=\int \! e^{2\pi i x\cdot\xi}a_j\ast_1\!\Psi_j (x,\xi)f(\xi)\,\textrm{d}\xi,
\end{equation*}
it is then easy to verify that
\begin{equation*}
R_jf(x)=\sum_{\substack{|k-j|\leq 1\\ k\geq 0}} Q_j Q_k^*f(x).
\end{equation*}
Lemma \ref{l2.3} (with $p=2$) and the Plancherel theorem  readily give $\|Q_j f\|_2\leq C\|f\|_2$. Then \eqref{lppdo3.2} follows from the simple fact  $\|Q_j\|_{L^2-L^2}=\|Q_j^*\|_{L^2-L^2}$.

Applying \eqref{lppdo3.1}, \eqref{lppdo3.2} and Cotlar's lemma in \cite[p. 280]{S93} yields the $L^2$ boundedness of the operator $SS^*$, as desired. \qed


%
%

%


\end{document}